\documentclass[12pt]{article}
\usepackage[utf8]{inputenc}
\usepackage{amsmath,amsthm,amssymb}
\usepackage{thmtools}
\usepackage{physics}
\usepackage{mathtools}
\usepackage{mathrsfs}
\usepackage{bbm}

\newtheorem{thm}{Theorem}[section]
\newtheorem{lem}[thm]{Lemma}
\newtheorem{cor}[thm]{Corollary}

\newtheorem{conj}[thm]{Conjecture}

\theoremstyle{definition}
\newtheorem{mydef}[thm]{Definition}
\newtheorem{rem}[thm]{Remark}
\newtheorem{ex}[thm]{Example}

\title{The NIEP is solvable by reality and finitely many polynomial inequalities}
\author{Jared J. L. ~Brannan \and Benjamin J.~Clark}
\date{July 2024}

\begin{document}

\maketitle

\begin{abstract}
The nonnegative inverse eigenvalue problem (NIEP) is shown to be solvable by the reality condition, spectrum equal to its conjugate, as well as by a finite union and intersection of polynomial inequalities. It is also shown that the symmetric NIEP and real NIEP form semi-algebraic sets and can therefore be solved just by a finite union and intersection of polynomial inequalities. An overview of ideas are given in how tools from real algebraic geometry may be applied to the NIEP and related sub-problems.
\end{abstract}

\section{Introduction}

The nonnegative inverse eigenvalue problem (NIEP) asks for necessary and sufficient conditions for a list of complex numbers to be the spectra of a nonnegative matrix, see \cite{johnson2018} and its references for an extensive survey of the topic.

The NIEP has been solved for $n \leq 4$. The solution for $n=4$ was first given by Meehan in \cite{meehan1998} which improved upon on existing necessary conditions for spectral bounds. An alternative solution for $n=4$ was given by Torre-Mayo \cite{torremayo2007} where instead of considering realizable spectra they considered coefficients of realizable polynomials, where a realizable polynomial is defined to be the characteristic polynomial for a nonnegative matrix. 

Motivated by Lowey and London in \cite{loewy1978}, Bharali and Holtz in \cite{bharali2008} searched for more necessary conditions for the NIEP by considering entire functions that preserve nonnegative matrices. In their introduction, they claim, without proof, that the NIEP can be solved using a finite number of polynomial inequalities following the spectra of nonnegative matrices form a semi-algebraic set. We prove that the coefficients of realizable characteristic polynomials form a semi-algebraic set, but the set of realizable spectra do not. However, we show that with the reality condition realizable spectra can still be solved by a finite number of polynomial inequalities.

The above polynomial inequalities form the certificate of what it means to solve the NIEP. Once those polynomial inequalities are derived, a potential spectra or characteristic polynomial can be checked in linear time with respect to the number of polynomials by simply evaluating each of the inequalities. This certificate points to future refinements in Meehan's solutions for $n=4$. 

While the realizable spectra of nonnegative matrices do not form a semi-algebraic set, the solutions to two major subproblems do. We show that both the set of spectra for the real nonnegative inverse eigenvalue problem (RNIEP) and symmetric nonnegative inverse eigenvalue problem (SNIEP) form semi-algebraic sets. 

%Egleston, Lenker, and Narayan in \cite{egleston2004} showed that the real and symmetric problems are different for $n \geq 5$. 

\section{Background and notation}

The set $\mathsf{M}_n$ is used for $n$ by $n$ real square matrices. The set restrictions $\mathsf{M}_n^{\geq 0}$ ($\mathsf{M}_n^{+})$ will be used to denote the $n$ by $n$ entry-wise nonnegative (positive) matrices.

Define $\mathscr{C}_n \subset \mathbb{R}^n$ as the set of coefficients for realizable characteristic polynomials. We order the tuple $(k_1, \dots, k_n) \in \mathscr{C}_n$ where $k_1$ corresponds with the $x^{n-1}$ term of the characteristic polynomial and $k_n$ with the constant term. The notation $\mathscr{C}^{\text{sym}}_n$ will be coefficients for realizable characteristic polynomials pulled from nonnegative symmetric matrices.

Denote $\mathscr{R}_n \subset \mathbb{C}^n$ to be the set of realizable spectra of $n$ by $n$ nonnegative matrices. Similar to above $\mathscr{R}_n^{\text{sym}}$ will be realizable spectra of symmetric $n$ by $n$ nonnegative matrices and $\mathscr{R}_n^{\text{real}}$ for real spectra of $n$ by $n$ nonnegative matrices. Finally $\mathscr{R}_{n,k}$ is defined to be the set of realizable spectra with $k$ conjugate pairs. 

Most of this paper is devoted to applying tools from real algebraic geometry to the NIEP and related problems. For results from real algebraic geometry we are pulling mainly from \cite{bochnak2013} with some of the complexity analysis pulled from \cite{basu2003}. We will use $R$ to denote a real closed field which is a field with an ordering and no nontrivial real algebraic extensions. A real closed field is a generalization of the field of real numbers, $\mathbb{R}$. For this paper, $\mathbb{R}$ will be the only real closed field considered.

\begin{mydef}
    A \textbf{semi-algebraic set} is a subset of $R^n$  of the form 
    \[
    \bigcup_{i=1}^s \bigcap_{j=1}^{r_i} \{x \in R^n \mid f_{i,j} *_{i,j} 0 \},
    \]
    where $f_{i,j} \in R[X_1, \dots, X_n]$ and $*_{i,j}$ is either $>$, $\geq$, or $=$, for $i=1,\dots, s$ and $j=1,\dots,r_i$.
    
    A \textbf{basic closed semi-algebraic} subset of $R^n$ is a set of the form 
    \[
    \{x \in R^n \mid f_1(x) \geq 0, \dots, f_s(x) \geq 0 \},
    \]
    where $f_1, \dots, f_s \in R[X_1, \dots, X_n]$. 
    
    A \textbf{basic open semi-algebraic} subset of $R^n$ is a set of the form 
    \[
    \{x \in R^n \mid f_1(x) > 0, \dots, f_s(x) > 0 \},
    \]
    where $f_1, \dots, f_s \in R[X_1, \dots, X_n]$.
\end{mydef}

\begin{thm}
    Let $A \subset R^n$ be an open (resp. closed) semi-algebraic set. Then $A$ is a finite union of basic open (resp. basic closed) semi-algebraic sets.
\end{thm}

Semi-algebraic sets are useful on their own as sets whose boundaries are defined by polynomial equations, but one of their main strengths is their stability through several types of natural operations. Most notably is if you project off a variable of a semi-algebraic set, the resulting set is still semi-algebraic. This result is given below.

\begin{mydef}
    Given a semi-algebraic set 
    \[
    S = \{ (x, y) \in R^n \times R : \Phi(x,y) \}
    \]
    where $\Phi(x,y)$ is some finite union and intersection of polynomial inequalities. The projection, $\pi: R^n \times R \rightarrow R^n$, of $S$ is defined by 
    \[
    \pi(S) = \{ x \in R^n : \exists y ~ \Phi(x,y) \}.
    \]
\end{mydef}

The projection of a variable in a semi-algebraic set moves that variable from a free variable to a quantified variable. The following theorem, originally given by Tarski and Seidenberg in \cite{tarski1998} and in \cite{seidenberg1954}, guarantees that the projected set is still semi-algebraic and the process of removing the quantifier from a system of polynomial inequalities is known as quantifier elimination.

\begin{thm}[Tarski-Seidenberg theorem] \label{thm:projection}
    Let $R$ be a real-closed field, let $\pi:R^{n} \times R^m \to R^n$ be the projection map, and let $S$ be a semi-algebraic set in $R^n \times R^m$. Then $\pi(S)$ is a semi-algebraic set in $R^n$.
\end{thm}

The unions of semi-algebraic sets make them generally unbounded, but always finite, in terms of the maximum number of polynomial inequalities that are needed to define them. However, if we consider basic semi-algebraic sets then the bound becomes sharp within a given dimension.

\begin{thm}[Theorem 6.5.1 from \cite{bochnak2013}] \label{thm:open_max_inequal}
    Let $V$ be an algebraic subset of $R^n$ of dimension of $d > 0$. Then every basic open semi-algebraic subset $U$ of $V$ can be defined by $d$ simultaneous strict polynomial inequalities: there exists $f_1, f_2, \dots, f_d \in R[t]$ such that 
    \[
    U = \{ x \in V \mid f_1(x) > 0, \dots, f_d(x) > 0 \}.
    \]
\end{thm}

\begin{thm}[Theorem 10.4.8 from \cite{bochnak2013}] \label{thm:closed_max_inequal}
    Let $V$ be an algebraic subset of $R^n$ of dimension of $d > 0$ and $T$ be a basic closed semi-algebraic subset of $V$. Then there exists $d(d+1)/2$ polynomial functions $f_1, \dots, f_{d(d+1)/2}$ on $V$ such that 
    \[
    T = \{x \in V | f_1(x) \geq 0, \dots, f_{d(d+1)/2}(x) \geq 0\}.
    \]
\end{thm}

To use semi-algebraic sets on the NIEP we need some well known results from matrix analysis. The following two definitions and following two remarks are pulled from \cite{horn_2012}.

\begin{mydef}
    Let $A \in \mathsf{M}_n$. The sum of its principal minors of size $k$ is denoted by $E_k(A)$.
\end{mydef}

\begin{mydef}
    The $k$th elementary symmetric function of $n$ complex numbers $\lambda_1, \dots, \lambda_n$, $k \leq n$, is
    \[
    S_k(\lambda_1, \dots, \lambda_n) = \sum_{1 \leq i_1 < \dots < i_k \leq n} \prod_{j=1}^k \lambda_{i_j}.
    \]
    For $A \in \mathsf{M}_n$ with eigenvalues $\lambda_1, \dots, \lambda_n$ we denote $S_k(A) = S_k(\lambda_1, \dots, \lambda_n)$.
\end{mydef}

\begin{rem} \label{rem:char_poly}
    Let $A \in \mathsf{M}_n$ have characteristic polynomial $p_A$, then 
    \[
    p_A(t) = t^n - S_1(A) t^{n-1} + \dots + (-1)^{n-1} S_{n-1}(A) t + (-1)^{n} S_{n}(A).
    \]
\end{rem}

\begin{rem} \label{rem:sym_ele}
    Let $A \in \mathsf{M}_n$. Then $S_k(A) = E_k(A)$ for each $k = 1, \dots, n$.
\end{rem}

Finally, we bring in some known necessary conditions for the NIEP. Let $\Lambda \in \mathbb{C}^n$ be a list of eigenvalues, then the $k$th moment of $\Lambda$ is
\[
s_k(\Lambda) = \sum_{j=1}^n \lambda_j^k.
\]

From \cite{johnson2018} some of the necessary conditions for the NIEP are
\begin{align*}
    0 &\leq s_1(\Lambda) \tag{Trace} \\
    (s_k(\Lambda))^m &\leq n^{m-1} s_{km}(\Lambda), ~ k,m = 1, 2, \dots \tag{JLL} \\
    \Lambda &= \overline{\Lambda} \tag{Reality} \\
    |\lambda_k| &\leq \lambda_1, ~ k=2,\dots, n, \tag{Perron} \\
\end{align*}

Notice that the Trace and JLL conditions are polynomial inequalities while reality and Perron are not. In particular, it is known that you can't determine complex conjugacy of a list through polynomial inequalities. This makes the reality condition unique as we will show below.

\section{The NIEP is solvable by finitely many polynomial inequalities}

Torre-Mayo et. al. in \cite{torremayo2007} considered an alternative posing of the NIEP which is ``Given $k_1,k_2, \dots, k_n$ real numbers, find necessary and sufficient conditions for the existence of a nonnegative matrix of order $n$ with characteristic polynomial $x^n + k_1 x^{n-1} + \dots + k_n$". A given list which is the coefficients of the characteristic polynomial of a nonnegative matrix is called realizable.

The outline to the following theorem comes from the introduction in \cite{bharali2008}. We formalize it here and give an explicit proof.

\begin{lem} \label{lem:realizable_coff}
    The set of realizable coefficients forms a semi-algebraic set.
\end{lem}

\begin{proof} 
    From remarks \ref{rem:char_poly} and \ref{rem:sym_ele}, the coefficients of the characteristic polynomial are determined by the sums of the principal minors. With this we can construct the non-projected semi-algebraic set for the coefficient form of the NIEP as 
    \[
    \{(k_1, \dots, k_n, A) \in \mathbb{R}^n \times \mathbb{R}^{n^2} : E_1(A) = k_1, \dots, E_n(A) = k_n, A \geq 0 \}.
    \]
    Now we project off $A$ using 
    Theorem \ref{thm:projection} to get the set 
    \[
    \{(k_1, \dots, k_n) \in \mathbb{R}^n : \exists A \in \mathbb{R}^{n^2} ~ E_1(A) = k_1, \dots, E_n(A) = k_n, A \geq 0 \}.
    \]
    By Theorem \ref{thm:projection} this set is semi-algebraic and this set is exactly the set of realizable coefficients.
\end{proof}

Using Remark \ref{rem:char_poly} and Lemma \ref{lem:realizable_coff} we want to now say that the realizable spectra for $n$ by $n$ nonnegative matrices define a semi-algebraic set. However, semi-algebraic sets are only defined over real closed fields. This gives the following theorem and corollary.

\begin{thm} \label{thm:rniep_pol}
    The set of realizable spectra of $n$ by $n$ nonnegative matrices with real spectra forms a semi-algebraic set.
\end{thm}

\begin{proof}
    From Lemma \ref{lem:realizable_coff} $\mathscr{C}_n$ forms the semi-algebraic set 
    \[
    \{(k_1, \dots, k_n) \in \mathbb{R}^n : \exists A \in \mathbb{R}^{n^2} ~ E_1(A) = k_1, \dots, E_n(A) = k_n, A \geq 0 \}.
    \]
    By Remark \ref{rem:sym_ele} the coefficients of the characteristic polynomial are also described by the elementary symmetric functions applied to the spectra. Thus we can substitute $k_i$ for $S_i(\sigma)$ for each $i = 1, \dots, n$. This gives that 
    \[
    \{\sigma \in \mathbb{R}^n : \exists A \in \mathbb{R}^{n^2} ~ E_1(A) = S_1(\sigma), \dots, E_n(A) = S_n(\sigma), A \geq 0 \}.
    \]
    The above set is made from polynomial inequalities and represents realizable real spectra of $n$ by $n$ nonnegative matrices.
\end{proof}

\begin{cor}
    The set of realizable spectra of $n$ by $n$ symmetric nonnegative matrices is a semi-algebraic set.
\end{cor}

\begin{proof}
    Follows Theorem \ref{thm:rniep_pol} and from noting that symmetric nonnegative matrices form a subset of nonnegative matrices with real spectra.
\end{proof}

The set of realizable spectra of $n$ by $n$ nonnegative matrices is not a semi-algebraic set since it is defined over the complex numbers which are not a real closed field. We can fix this with the necessary reality NIEP condition and by noting that symmetric polynomials map complex conjugates to real numbers. With this connection we can show that the set of spectra for $n$ by $n$ nonnegative matrices can be described by polynomial inequalities along with the reality condition.

\begin{mydef} \label{def:conj_to_real_map}
    Let $U_k \subset \mathbb{C}^n$ such that the first $n-2k$ entries of the tuple are real and the last $2k$ entries are $k$ conjugate pairs. Then define the bijection $\Phi_k: U_k \rightarrow \mathbb{R}^n$ as 
    \begin{align*}
        \Phi_k(x_1, \dots, x_{n-2k}, y_1 + iz_1, y_1 - iz_1, \dots, y_k + iz_k, y_k - iz_k&) \\
        = (x_1, \dots, x_{n-2k}, y_1, z_1, \dots, y_k, z_k&).
    \end{align*}
\end{mydef}

\begin{mydef}
    For a fixed $n$, define the modified elementary symmetric functions as $S_{j,k}: \mathbb{R}^n \rightarrow \mathbb{R}$ constructed as
    \[
    S_{j,k}(x) = S_j(\phi_k^{-1}(x)).
    \]
\end{mydef}

\begin{ex}
    Let $n = 3$ and $\sigma = (\lambda_1, \lambda_2, \lambda_3) \in \mathbb{R}$, then there are $6$ modified elementary symmetric functions 
    \begin{align*}
        S_{1,0}(\sigma) &= \lambda_1 + \lambda_2 + \lambda_3 \\
        S_{2,0}(\sigma) &= \lambda_1 \lambda_2 + \lambda_1 \lambda_3 + \lambda_2 \lambda_3 \\
        S_{3,0}(\sigma) &= \lambda_1 \lambda_2 \lambda_3 \\
        S_{1,1}(\sigma) &= \lambda_1 + 2\lambda_2 \\
        S_{2,1}(\sigma) &= 2\lambda_1 \lambda_2 + \lambda_2^2 + \lambda_3^2 \\
        S_{3,1}(\sigma) &= \lambda_1\lambda_2^2 + \lambda_1 \lambda_3^2.
    \end{align*}
    Notice that $S_{k,0} = S_k$ and that for $j > 0$, $S_{k,j}$ is no longer a symmetric function.
\end{ex}

\begin{thm}
    Let $\sigma \in \mathbb{C}^n$, then $\sigma$ can be verified realizable by the reality condition and the inclusion into one of $\lfloor (n+1)/2 \rfloor$ semi-algebraic sets.
\end{thm}

\begin{proof}
    Let $\sigma \in \mathbb{C}^n$ such that $\sigma$ satisfies the reality condition, then we can write $\sigma$ as a list of $n-2k$ real numbers followed by $k$ conjugate pairs. Using $\Phi_k: U_k \rightarrow \mathbb{R}^n$ from definition \ref{def:conj_to_real_map} we can map $\Phi_k(\sigma) = \tau$. We can now proceed in a similar way to Theorem \ref{thm:rniep_pol} by constructing
    \[
    \{x \in \mathbb{R}^n : \exists A \in \mathbb{R}^{n^2} ~ E_1(A) = S_{1,k}(x), \dots, E_n(A) = S_{n,k}(x), A \geq 0 \}.
    \]
    which represents the set of all realizable spectra with $k$ conjugate pairs. Thus $\sigma$ is realizable if and only if $\sigma$ satisfies the reality condition and if $\tau$ is in the above set. 
\end{proof}

In general moving between the closure and the interior of a semi-algebraic set is not as simple as changing inequalities and strict inequalities, see \cite{bochnak2013} for examples. In the case of the semi-algebraic sets that make up the NIEP this is the case as is shown below.

\begin{thm} \label{thm:set_equal_closint}
    Let $\mathscr{R}_{n,k}$ be the semialgebraic set of realizable spectra with $k$ complex conjugates, then 
    \[
    \text{cl}\left(\text{int}\left(\mathscr{R}_{n,k}\right)\right) = \mathscr{R}_{n,k}.
    \]
\end{thm}

\begin{proof}
    First note that the set of realizable spectra is closed, this gives that $\text{cl}\left(\text{int}\left(\mathscr{R}_{n,k}\right)\right) \subseteq \mathscr{R}_{n,k}$. 
    
    Let $\sigma$ be a realizable spectra on the boundary of $\mathscr{R}_{n,k}$ with associated matrix $A \in \mathsf{M}^{\geq 0}_n$ and let $\mathbbm{1}_s \in \mathsf{M}_n$ be the all ones matrix times some number $s$. Now consider the sequence of strictly positive matrices $\{A+\mathbbm{1}_{1/s}\}_{s=1}^\infty$. This sequence converges to the matrix $A$, but at every point in the sequence we are in the interior since $A + \mathbbm{1}_{1/s}$ is positive for every $s \in \mathbb{N}^+$. By the continuity of the spectra of a matrix the sequence of spectras associated with $\{A+\mathbbm{1}_{1/s}\}_{s=1}^\infty$ will converge to $\sigma$. Thus the result follows from $\sigma$ being an arbitrary boundary point of $\mathscr{R}_{n,k}$.
\end{proof}

\section{Matrices that lie on the boundary}

An important related problem to solving the NIEP is to get a set of matrices whose spectra lie on the boundary of the NIEP. In a certain view these matrices form the ``spectrally least nonnegative" matrices. Using semi-algebraic projections we can find these boundary matrices by taking the solution of the NIEP and placing that back in the embedded semi-algebraic set. Then we project off the spectral values to get the matrices that lie on the boundary.

\begin{ex}
    Consider $\sigma = (1,-1)$ and a generic 
    \[
    A = \begin{bmatrix}
        a & b \\ c & d
    \end{bmatrix}.
    \]
    From the construction of the embedded semi-algebraic set for the RNIEP of $n=2$ we have that 
    \begin{align*}
        a + d &= 0 \\
        ad - bc &= -1 \\
        a,b,c,d &\geq 0
    \end{align*}
    which implies $a = d = 0$ since we are dealing with nonnegative matrices. Thus $bc = 1$, giving that the set of matrices with the given spectra are 
    \[
    \begin{bmatrix}
        0 & b \\ 1/b & 0
    \end{bmatrix}
    \]
    where $b > 0$.
\end{ex}

This process can also provide a reliable method for finding realizable spectra. Given a spectra $\sigma$ with $k$ conjugate pairs we can apply the quantifier elimination process to determine if the following statement is true 
\[
\exists A \in \mathbb{R}^{n^2} ~ E_1(A) = S_{1,k}(\sigma), \dots, E_n(A) = S_{n,k}(\sigma), A \geq 0.
\]

As we will mention below the general algorithm for quantifier elimination is computationally infeasible. But the process of checking if there exists any matrix that satisfies the above equation can be done much faster, though still in exponential time. This leaves the problem tractable for low dimensional cases, see \cite{basu2003} for more information.

\section{Concluding remarks}

In our view knowing the existence of polynomial inequalities is an exciting result. A further aspect of real algebraic geometry is cylindrical algebraic decomposition (CAD) or Collins algorithm, see \cite{collins1975} for the original algorithm and \cite{basu2003} for a detailed modern refinement of algorithms in real algebraic geometry. CAD gives an algorithm to directly compute the polynomials through projections. This algorithm gives a method for solving the NIEP for any value of $n$. In practice CAD cannot be used directly for larger problems. The algorithm complexity is doubly exponential both with respect to the number of polynomials before the projection and the max degrees of those polynomials. 

The set of realizable spectra are closed when working nonnegative matrices and open when working with positive matrices. With this we can restrict ourselves to either basic closed or basic open semi-algebraic sets depending on whether positive or nonnegative matrices are being considered. 

One problem with the semi-algebraic approach is bounding the number of polynomial inequalities needed. The bound for the number of polynomial inequalities that can be added with projection is doubly exponential \cite{basu2003}. This leads to an incalculable worst case when we project off the needed $n^2$ variables. To get a handle on the number of polynomials needed we can use Theorems \ref{thm:open_max_inequal} and \ref{thm:closed_max_inequal}. To use these we need to bound the number of basic semi-algebraic sets that make up the space. With this we get that the number of polynomial inequalities that make up $\mathscr{R}_{n,k}$ is bounded by $n(n+1)/2$ times the number of basic closed semi-algebraic sets. Using Theorem \ref{thm:set_equal_closint} we can actually use Theorem \ref{thm:open_max_inequal} on $\mathscr{R}_{n,k}$ to bring the bound down to $n$ times the number of basic closed semi-algebraic sets. Similarly if we restrict to positive matrices the bound becomes $n$ times the number of basic open semi-algebraic sets.

The SNIEP and RNIEP form connected semi-algebraic sets, but connected does not necessarily imply basic. Based on the symmetries in the cone we are projecting we conjecture the following.

\begin{conj} \label{conj:one_real_basic}
    The the set of real realizable spectra for nonnegative matrices forms a basic semi-algebraic set. 
\end{conj}

\begin{conj}
    The the set of realizable spectra for symmetric nonnegative matrices forms a basic semi-algebraic set.
\end{conj}

If the above conjectures are proven, then by Theorem \ref{thm:open_max_inequal} the number of polynomial inequalities needed is less than or equal to $n$. Conjecture \ref{conj:one_real_basic} would also point towards the following conjecture.

\begin{conj}
    The set of realizable spectra for nonnegative matrices is solvable with the reality condition and $\lfloor (n+1)/ 2 \rfloor$ unions of $n$ polynomial inequalities. 
\end{conj}

\bibliographystyle{abbrv}
\bibliography{Papers/polynomial_inequalities/polynomial_inequalities_solve_niep}

\end{document}